\documentclass[a4paper]{amsart}

\usepackage{amsmath,amsbsy,comment}
\usepackage[psamsfonts]{amssymb}
\usepackage{hyperref}
\hypersetup{%
  bookmarksnumbered=true,%
  bookmarks=true,%
}
\title{Morse--Bott inequalities for manifolds with boundary}

\author{Ryuma Orita} 
\address{Graduate School of Mathematical Sciences, the University of Tokyo, 3-8-1 Komaba, Meguro-ku, Tokyo, 153-0041, Japan}
\email{orita@ms.u-tokyo.ac.jp}
\urladdr{https://sites.google.com/site/oritaryuma/}


\newtheorem{theorem}{Theorem}[section]
\newtheorem{lemma}[theorem]{Lemma}
\newtheorem{proposition}[theorem]{Proposition}
\newtheorem{corollary}[theorem]{Corollary}

\theoremstyle{definition}
\newtheorem{definition}[theorem]{Definition}
\newtheorem*{remark}{Remark}

\newcommand{\Int}{\mathop{\mathrm{Int}}\nolimits}
\newcommand{\Cr}{\mathop{\mathrm{Crit}}\nolimits}
\newcommand{\ind}{\mathop{\mathrm{ind}}\nolimits}
\newcommand{\rank}{\mathop{\mathrm{rank}}\nolimits}
\newcommand{\grad}{\mathop{\mathrm{grad}}\nolimits}
\newcommand{\Ker}{\mathop{\mathrm{Ker}}\nolimits}

\newcommand{\var}{\mathop{\mathrm{var}}\nolimits}
\newcommand{\tp}{\mathop{\mathrm{top}}\nolimits}
\newcommand{\1}{\mbox{1}\hspace{-0.25em}\mbox{l}}

\numberwithin{equation}{section}

\begin{document}

\maketitle


\begin{abstract}
In the present paper, we define Morse--Bott functions on manifolds with boundary which are generalizations of Morse functions
and show Morse--Bott inequalities for these manifolds.
\end{abstract}


\section{Introduction}

In the early 1930s, M. Morse \cite{Morse} showed the well-known Morse inequalities which describe the relationship
between the Betti numbers of manifolds and the number of critical points of functions.
Many authors found ``Morse inequalities'' for manifolds with non-empty boundary.
Recently, motivated by the Floer homology, M. Akaho \cite{Akaho} constructed the Morse complex
(which is derived from Witten's work \cite{Witten}) by setting specific conditions on the boundary.
F. Laudenbach \cite{Laudenbach} constructed the Morse complex
by introducing the pseudo-gradient vector fields adapted to the boundary which control flow lines near the boundary
and obtained Morse inequalities for manifolds with boundary.

In 1954, R. Bott \cite{Bott} generalized the Morse theory for
functions which have degenerate critical points under some assumptions.
Indeed, he \cite{Bott2, Bott3, Bott4} established degenerate Morse inequalities (called Morse--Bott inequalities).
After that, J.-M. Bismut \cite{Bismut}, B. Helffer and J. Sj\"{o}strand \cite{HS} gave different proofs of these inequalities.
Recently, A. Banyaga and D. Hurtubise \cite{BH} published a proof
by describing an explicit perturbation of a given Morse--Bott function (called the perturbation technique)
under the assumption that negative normal bundles are all orientable.
Such functions appear when the domain manifold $M$ is equipped with the action of a compact Lie group $G$.
In that case, the critical set of any $G$-invariant smooth function is a disjoint union of finitely many closed submanifolds.

In the present paper, by using the definition introduced by F. Laudenbach \cite{Laudenbach},
we define Morse--Bott functions on manifolds with boundary which are generalizations of Morse functions.
Then we formulate Morse--Bott inequalities for manifolds with boundary (Theorem \ref{theorem:main})
and show them by using the perturbation technique.


\section{Preliminaries: Morse homology with local coefficients}\label{section:preliminaries}

In this section, we define the Morse homology with local coefficients and state the Morse homology theorem and the Morse inequalities for closed manifolds.
Our main reference is \cite[Subsection 7.2]{O}.

Let $M$ be an $m$-dimensional connected closed manifold and $f\colon M\to\mathbb{R}$ a Morse function on $M$.
Let $R$ be a ring and $\mathcal{L}$ a local system of $R$-modules over $M$.
We define the Morse complex $\left(C_{\ast}(f;\mathcal{L}),\partial_{\ast}^{f;\mathcal{L}}\right)$ for $f$ with local coefficients in $\mathcal{L}$.
For each $k=0,1,\ldots,m$, the $k$-chain group is defined by
\[
	C_k(f;\mathcal{L})=\bigoplus_{p\in\Cr_k(f)}\mathcal{L}_p,
\]
where $\mathcal{L}_p$ is the fiber of $\mathcal{L}$ over $p\in M$.
The boundary operator $\partial_k^{f;\mathcal{L}}\colon C_k(f;\mathcal{L})\to C_{k-1}(f;\mathcal{L})$ is defined by
\[
	\partial_k^{f;\mathcal{L}}(s\otimes p)=%
	\sum_{q\in\Cr_{k-1}(f)}\left(\sum_{\gamma\in\mathcal{M}_f(p,q)}n_f(\gamma)\Phi_{\tilde{\gamma}}(s)\right)\otimes q,
\]
where $p\in\Cr_k(f)$, $s\in\mathcal{L}_p$,
the set $\mathcal{M}_f(p,q)=W_f^u(p)\cap W_f^s(q)/\mathbb{R}$ is the moduli space of unparametrized (minus) gradient flow lines of $f$ from $p$ to $q$,
$n_f(\gamma)$ is the sign of $\gamma$ determined by orientations of the unstable manifolds for $f$,
the path $\tilde{\gamma}\colon [0,1]\to M$ is a reparametrization of $\gamma\colon\mathbb{R}\to M$ such that $\tilde{\gamma}(0)=p$ and $\tilde{\gamma}(1)=q$,
and $\Phi_{\tilde{\gamma}}\colon\mathcal{L}_p\to\mathcal{L}_q$ is the $R$-isomorphism associated with $\tilde{\gamma}$.

\begin{theorem}[Morse homology theorem \cite{O}]\label{theorem:mht}
The pair $\left(C_{\ast}(f;\mathcal{L}),\partial_{\ast}^{f;\mathcal{L}}\right)$ is a chain complex,
i.e., $\partial_{\ast}^{f;\mathcal{L}}$ is well-defined and $\partial_{\ast-1}^{f;\mathcal{L}}\circ\partial_{\ast}^{f;\mathcal{L}}=0$.
Moreover, its homology is isomorphic to the singular homology of $M$ with local coefficients in $\mathcal{L}$.
Namely, for every $k=0,1,\ldots,m$ we have
\[
	H_k\left(C_{\ast}(f;\mathcal{L}),\partial_{\ast}^{f;\mathcal{L}}\right)\cong H_k(M;\mathcal{L}).
\]
\end{theorem}

We state the Morse inequalities for closed manifolds.

\begin{definition}
The \textit{Morse counting polynomial} of $f$ is defined to be
\[
	\mathcal{M}_t(f)=\sum_{p\in\Cr(f)}t^{\ind_fp},
\]
and the \textit{Poincar\'e polynomial} $\mathcal{P}_t(M;\mathcal{L})$ of $M$ with local coefficients in $\mathcal{L}$ by
\[
	\mathcal{P}_t(M;\mathcal{L})=\sum_{k=0}^m \rank_R H_k(M;\mathcal{L})\,t^k.
\]
\end{definition}

\begin{theorem}[Morse inequalities for closed manifolds]\label{theorem:mi}
Let $M$ be a connected closed manifold and $f$ a Morse function on $M$.
Let $R$ be a ring and $\mathcal{L}$ a local system of $R$-modules over $M$.
Then there exists a polynomial $\mathcal{R}(t)$ with non-negative coefficients such that
\[
	\mathcal{M}_t(f)-\frac{1}{\rank{\mathcal{L}}}\mathcal{P}_{t}(M;\mathcal{L})=(1+t)\mathcal{R}(t).
\]
\end{theorem}

The proof is straightforward by applying Theorem \ref{theorem:mht}.


\section{Main Theorem}\label{section:maintheorem}

In this section, we define Morse--Bott functions on manifolds with boundary in the sense of Laudenbach
and state our main theorem (Theorem \ref{theorem:main}).
Let $M$ be an $m$-dimensional compact manifold with boundary.

\begin{definition}\label{definition:MBwb}
A $C^{\infty}$-function $f\colon M\to \mathbb{R}$ is called a \textit{Morse--Bott function} if $f$ satisfies the following conditions:
\begin{enumerate}
	\item The set of critical points $\Cr (f)=\{\, p\in M\mid df(p)=0\, \}$ is a disjoint union of connected submanifolds of the interior $\Int M$
	and each connected component of $\Cr (f)$ is non-degenerate in the sense of Bott (see \cite{Bott}).
	\item The restriction $f|_{\partial M}\colon \partial M\to \mathbb{R}$ is a Morse--Bott function on $\partial M$.
\end{enumerate}
\end{definition}

Let $f\colon M\to \mathbb{R}$ be a Morse--Bott function.
Then the critical submanifolds of $f|_{\partial M}$ are divided into two types:

\begin{definition}
A connected critical submanifold $C\subset \Cr(f|_{\partial M})$ is said to be of \textit{type $N$} (resp.\ \textit{type $D$}) if
for some $p\in C$, hence for all $p\in C$, $p$ is of type $N$ (resp.\ type $D$), i.e.,
$\left<df(p), n(p)\right>$ is negative (resp.\ positive)
where $n(p)\in T_{p}M$ is an outward normal vector to the boundary at $p$ (see \cite{Laudenbach}).
\end{definition}

We denote the critical point sets by
\[
	I(f)=\Cr(f),\quad N(f)=\{\, p\in \Cr(f|_{\partial M})\mid p\ \text{is of type $N$}\, \}
\]
and
\[
	D(f)=\{\, p\in \Cr(f|_{\partial M})\mid p\ \text{is of type $D$}\, \}.
\]
Let $C_{j}$ ($j=1,\ldots,\ell$), $\Gamma_{s}$ ($s=1,\ldots,\ell_N$) and $\Delta_{u}$ ($u=1,\ldots,\ell_D$) be the connected components
of $I(f)$, $N(f)$ and $D(f)$, respectively.
We prepare the following notation:
For $j=1,\ldots,\ell$,\: $s=1,\ldots,\ell_N$ and $u=1,\ldots,\ell_D$, let
\begin{align*}
	c_{j}&=\dim C_{j},& d_{s}^{N}&=\dim \Gamma_{s},& d_{u}^{D}&=\dim \Delta_{u},\\
	\lambda_{j}&=\ind_{f} C_{j},& \mu_{s}^{N}&=\ind_{f|_{\partial M}} \Gamma_{s},& \mu_{u}^{D}&=\ind_{f|_{\partial M}} \Delta_{u}.
\end{align*}

For a non-degenerate critical submanifold $C$ of the Morse--Bott function $f$, denote by $o(\nu^-C)$ the orientation bundle of the negative normal bundle $\nu^-C$,
where $\nu^-C$ is the maximal subbundle of the normal bundle of $C$ in $M$ such that the eigenvalues of the Hessian of $f$ are all negative.
We think of $o(\nu^-C)$ as a local system of $\mathbb{Z}$-modules of rank one over $C$.

\begin{definition}\label{definition:MBpoly}
The \textit{Morse--Bott counting polynomial of type $N$} of $f$ is defined to be
\[
	\mathcal{MB}^{N}_{t}(f)=\sum_{j=1}^{\ell} \mathcal{P}_{t}\bigl(C_{j};o(\nu^-C_j)\bigr)\, t^{\lambda_{j}} + \sum_{s=1}^{\ell_N} \mathcal{P}_{t}\bigl(\Gamma_{s};o(\nu^-\Gamma_{s})\bigr)\, t^{\mu_{s}^{N}},
\]
and the \textit{Morse--Bott counting polynomial of type $D$} of $f$ is defined to be
\[
	\mathcal{MB}^{D}_{t}(f)=\sum_{j=1}^{\ell} \mathcal{P}_{t}\bigl(C_{j};o(\nu^-C_j)\bigr)\, t^{\lambda_{j}} + \sum_{u=1}^{\ell_D} \mathcal{P}_{t}\bigl(\Delta_{u};o(\nu^-\Delta_{u})\bigr)\, t^{\mu_{u}^{D}+1}.
\]
\end{definition}

Our main result is the following theorem:

\begin{theorem}[Morse--Bott inequalities for manifolds with boundary]\label{theorem:main}
Let $M$ be a compact manifold with boundary and $f$ a Morse--Bott function on $M$.
Then there exists a polynomial $\mathcal{R}(t)$ with non-negative integer coefficients such that
\[
	\mathcal{MB}^N_t(f)-\mathcal{P}_t(M;\mathbb{Z})=(1+t)\mathcal{R}(t).
\]
\end{theorem}

\begin{corollary}
Suppose that $M$, the critical submanifolds of $f$ and their negative normal subbundles are all oriented.
Then there exists a polynomial $\mathcal{R}(t)$ with non-negative integer coefficients such that
\[
	\mathcal{MB}^{D}_{t}(f) - \mathcal{P}_{t}(M, \partial M;\mathbb{Z})= (1+t)\mathcal{R}(t).
\]
\end{corollary}

\begin{proof}
Since $f$ is Morse--Bott, so is $-f$.
The critical point sets of $f$ are decomposed into critical submanifolds as follows:
\[
	I(f)= \bigsqcup_{j=1}^{\ell}C_{j},\quad
	N(f)= \bigsqcup_{s=1}^{\ell^{N}} \Gamma_{s},\quad
	D(f)= \bigsqcup_{u=1}^{\ell^{D}} \Delta_{u}.
\]
Then the critical point sets of $-f$ are
\[
	I(-f)= \bigsqcup_{j=1}^{\ell}C_{j},\quad
	N(-f)= \bigsqcup_{u=1}^{\ell^{D}} \Delta_{u},\quad
	D(-f)= \bigsqcup_{s=1}^{\ell^{N}} \Gamma_{s}.
\]
Moreover, for $j=1,\ldots,\ell$ and $u=1,\ldots,\ell_D$ we have
\begin{align*}
	\ind_{-f}C_j&=\dim M-\dim C_j-\ind_f C_j=m-c_j-\lambda_j,\\
	\ind_{-f|_{\partial M}}\Delta_u&=\dim \partial M-\dim \Delta_u-\ind_{f|_{\partial M}} \Delta_u=m-d_u^D-(\mu_u^D+1).
\end{align*}
Theorem \ref{theorem:main} implies that there exists a polynomial $\mathcal{Q}(t)$ with non-negative integer coefficients such that
\[
	\mathcal{MB}^{N}_{t}(-f) - \mathcal{P}_{t}(M;\mathbb{Z})= (1+t)\mathcal{Q}(t).
\]
Since the negative normal bundles are all oriented, $\mathcal{MB}^{N}_{t}(-f)$ is computed as
\begin{align*}
	\mathcal{MB}_{t}^{N}({-f})
	&= \sum_{j=1}^{\ell} \mathcal{P}_{t}(C_{j};\mathbb{Z})\, t^{\ind_{-f} C_{j}} + \sum_{u=1}^{\ell^D} \mathcal{P}_{t}(\Delta_{u};\mathbb{Z})\, t^{\ind_{-f|_{\partial M}} \Delta_{u}} \\
	&= \sum_{j=1}^{\ell} \mathcal{P}_{t}(C_{j};\mathbb{Z})\, t^{m-c_j-\lambda_j} + \sum_{u=1}^{\ell^D} \mathcal{P}_{t}(\Delta_{u};\mathbb{Z})\, t^{m-d_u^D-(\mu_u^D+1)} \\
	&= t^{m} \left\{ \sum_{j=1}^{\ell} \mathcal{P}_{t}(C_{j};\mathbb{Z})\, t^{-c_{j}} \left(\frac{1}{t} \right)^{\lambda_{j}}
		+ \sum_{u=1}^{\ell^D} \mathcal{P}_{t}(\Delta_{u};\mathbb{Z})\, t^{-d_{u}^D} \left(\frac{1}{t} \right)^{\mu_{u}^D+1} \right\} \\
	&= t^{m} \left\{ \sum_{j=1}^{\ell} \mathcal{P}_{1/t}(C_{j};\mathbb{Z}) \left(\frac{1}{t} \right)^{\lambda_{j}}
		+ \sum_{u=1}^{\ell^D} \mathcal{P}_{1/t}(\Delta_{u};\mathbb{Z}) \left(\frac{1}{t} \right)^{\mu_{u}^D+1} \right\}\\
	&= t^{m} \mathcal{MB}_{1/t}^D(f).
\end{align*}
Here we used the fact that the critical submanifolds are all oriented and hence we have
\begin{align*}
	\mathcal{P}_{t}(C_{j};\mathbb{Z})\, t^{-c_{j}}
	&= \left(\sum_{k=0}^{c_{j}} \rank_{\mathbb{Z}}H_{k}(C_{j};\mathbb{Z})\, t^{k} \right) t^{-c_{j}}
	= \sum_{k=0}^{c_{j}} \rank_{\mathbb{Z}}H_{k}(C_{j};\mathbb{Z}) \left(\frac{1}{t} \right)^{c_{j}-k} \\
	&= \sum_{k=0}^{c_{j}} \rank_{\mathbb{Z}}H_{c_{j}-k}(C_{j};\mathbb{Z}) \left(\frac{1}{t} \right)^{c_{j}-k}
	= \mathcal{P}_{1/t}(C_{j};\mathbb{Z}).
\end{align*}
Similarly we have $\mathcal{P}_{t}(\Delta_{u};\mathbb{Z})\, t^{-d_{u}^D}=\mathcal{P}_{1/t}(\Delta_{u};\mathbb{Z})$.

Therefore we obtain
\begin{align*}
	t^{m} \mathcal{MB}^D_{1/t}(f) - \mathcal{P}_{t}(M;\mathbb{Z})&= (1+t)\mathcal{Q}(t), \\
	\frac{1}{t^{m}} \mathcal{MB}^D_{t}(f) - \mathcal{P}_{1/t}(M;\mathbb{Z})&= \left(1+\frac{1}{t} \right)\mathcal{Q} \left(\frac{1}{t} \right), \\
	\mathcal{MB}^D_{t}(f) - t^{m} \mathcal{P}_{1/t}(M;\mathbb{Z})&= (1+t)t^{m-1} \mathcal{Q} \left(\frac{1}{t} \right).
\end{align*}
Note that the polynomial $t^{m-1} \mathcal{Q} (1/t)$ is a polynomial with non-negative integer coefficients.
Since $M$ is oriented, Lefschetz duality shows that
\begin{align*}
	t^{m} \mathcal{P}_{1/t}(M;\mathbb{Z})&= t^{m} \sum_{k=0}^{m}\ \rank_{\mathbb{Z}}H_{k}(M;\mathbb{Z}) \left(\frac{1}{t} \right)^{k} \\
	&= \sum_{k=0}^{m}\ \rank_{\mathbb{Z}}H^{m-k}(M,\partial M; \mathbb{Z})\, t^{m-k}
	= \mathcal{P}_{t}(M,\partial M).\qedhere
\end{align*}
\end{proof}

The perturbation technique of Section \ref{section:3} is nothing but the \textit{morsification} of $f$ along its critical set.
Then, one reduces to the known case of a Morse function.
Of course, there is still some work to do, since the main theorem is stated in terms of the Bott critical sets.
This is done in the last four pages of the paper.


\section{Proof of Main Theorem}\label{section:3}

Let $M$ be an $m$-dimensional compact manifold with boundary and $f$ a Morse--Bott function on $M$.

\subsection{Adapted pseudo-gradient vector fields of Bott-type}

Let $X\in \mathfrak{X}(M)$ be a vector field on $M$, i.e.,
a smooth section $X\colon M\to TM$ in the tangent bundle.
For any zero point $p$ of $X$, the differential $DX$ of the smooth map $X$ defines a vector field $DX(p)\colon T_p M\to T_{X(p)}TM$.
Let $\phi\colon T_{X(p)}TM\to T_pM$ be an isomorphism.
We call $X_p^{\mathrm{lin}}=\phi\circ DX(p)\colon T_pM\to T_pM$ the \textit{linear part} of $X$ at $p$.
Let $C$ be a submanifold consisting of zero points of $X$.
$C$ is called a \textit{zero submanifold}.
The tangent bundle $\left.TM\right|_{C}$ splits as
$\left.TM\right|_{C}= TC\oplus \nu C$
where $\nu C$ is the normal bundle of $C$ in $M$.
For all $p\in C$ and $V\in T_{p}C$, we have $DX(p)(V)=0$.
Hence the linear part $X_p^{\mathrm{lin}}$ of $X$ at $p$ induces the linear transformation
$\bigl(X_p^{\mathrm{lin}}\bigr)^{\nu}\colon \nu_{p}C\to \nu_{p}C$.

\begin{definition}[\cite{AR}]
A zero submanifold $C$ of $X$ is said to be \textit{transversely hyperbolic}
if for each $p\in C$, the linear transformation $\bigl(X_p^{\mathrm{lin}}\bigr)^{\nu}$
has no pure imaginary eigenvalues.
\end{definition}

\begin{definition}\label{definition:pseudoBott}
A vector field $X\in \mathfrak{X}(M)$ is called a \textit{pseudo-gradient vector field of Bott-type}
for the Morse--Bott function $f$ \textit{adapted} to the boundary
if $X$ satisfies the following conditions:
\begin{enumerate}
	\item $Xf<0$ except on the interior critical submanifolds and the boundary critical submanifolds of type $N$ of $f$.
	\item $X$ points inwards along the boundary except a neighborhood of the boundary critical submanifolds of type $N$
			where it is tangent to $\partial M$.
	\item For each interior critical submanifold $C\subset I(f)$, $C$ is a transversely hyperbolic zero submanifold of $X$, and
			for all $p\in C$, the Hessian of the Lie derivative $X\cdot f$ restricted to the normal space $\nu_p C$
			is negative definite.
	\item For each boundary critical submanifold $\Gamma\subset N(f)$ of type $N$ and $\gamma\in \Gamma$,
			there exist local coordinates $x=(r,u,y)\in [0,1)\times \mathbb{R}^{\dim \Gamma}\times \mathbb{R}^{m-\dim \Gamma-1}$
			of $M$ around $\gamma$ such that $M=\{r\geq 0 \}$
			and $f(x)=f(\Gamma)+q(y)+r$ where $q$ is a non-degenerate quadratic form.
			Moreover, in these coordinates, $X$ is tangent to the boundary, $\Gamma$ is a transversely hyperbolic zero submanifold of $X$, and
			the Hessian of the Lie derivative $X\cdot f$ restricted to the normal bundle $\nu \Gamma$ is negative definite.
\end{enumerate}
\end{definition}

\begin{remark}
The conditions (1) and (2) are exactly the same as \cite{Laudenbach}.
However, in our setting, since the critical point set might be positive dimensional,
the conditions (3) and (4) are certain generalizations of those of \cite{Laudenbach}.
\end{remark}

\begin{proposition}\label{proposition:pseudoBott}
For any Morse--Bott function $f$, there exists an adapted pseudo-gradient vector field of Bott-type $X\in \mathfrak{X}(M)$ for $f$.
\end{proposition}

\begin{proof}
The proof is similar to that of \cite[Proposition of Subsection 2.1]{Laudenbach}.
The existence of local coordinates of (4) follows from the Morse--Bott Lemma \cite{Bott}.
\end{proof}


\subsection{Perturbation technique \cite{BH}}\label{subsection:pt}

In this subsection, we use the same notation as in Section \ref{section:maintheorem}.
We identify a collar neighborhood of the boundary $\partial M$ with $[0,1)\times \partial M$.
Let $r$ be the standard coordinate of $[0,1)$.
Fix an adapted pseudo-gradient vector field of Bott-type $X_{f}\in \mathfrak{X}(M)$
for the Morse--Bott function $f$ (Proposition \ref{proposition:pseudoBott}).
For all $j=1,\ldots,\ell$ and $s=1,\ldots,\ell_N$, we choose open neighborhoods $U_{j}\subset \Int M$ of $C_{j}$
and $U_{s}^{N}\subset \partial M$ of $\Gamma_{s}$ such that
\begin{equation}
	\left.X_{f}\right|_{U_{j}}= -\left.(\grad f)\right|_{U_{j}}
	\quad \text{and}\quad \left.X_{f}\right|_{U_{s}^{N}}= -\left.(\grad f|_{\partial M})\right|_{U_{s}^{N}}.\label{eq:1}
\end{equation}
By the condition (4) of Definition \ref{definition:pseudoBott},
we can choose a positive real number $0<\delta<1/2$ small enough so that for any $s=1,\ldots,\ell_N$ and $\gamma\in \Gamma_{s}$,
there exist local coordinates
$(r,u,y)\in [0,1)\times \mathbb{R}^{\dim \Gamma_{s}}\times \mathbb{R}^{m-\dim \Gamma_{s}-1}$ of $M$ around $\gamma$
and a non-degenerate quadratic form $q_{s}$ such that
\begin{align}
	f|_{[0,2\delta)\times U_{s}^{N}}&= f(\Gamma_{s})+q_{s}(y)+r,\label{eq:2}\\
	\left.X_{f}\right|_{[0,2\delta)\times U_{s}^{N}}&= -\left.(\grad f|_{\partial M})\right|_{U_{s}^{N}} - r\frac{\partial}{\partial r}.\label{eq:3}
\end{align}
For all $j$ and $s$, let $T_{j}\subset \Int M$ be a tubular neighborhood of $C_{j}$ of radius $\delta_{j}$ and
$T_{s}^{N}\subset \partial M$ a tubular neighborhood of $\Gamma_{s}$ which satisfy the following conditions:
\begin{enumerate}
	\item For each $j$ and $s$, we have $T_{j}\subset U_{j}$ and $T_{s}^{N}\subset U_{s}^{N}$.
	\item For each $j$ and $s$, $T_{j}$ and $T_{s}^{N}$ are contained in the union of the charts from the Morse--Bott Lemma \cite{Bott}.
	\item For distinct $i$ and $j$, we have $T_{i}\cap T_{j}=\emptyset$ and $T_{i}^{N}\cap T_{j}^{N}=\emptyset$.
	\item For every pseudo-gradient flow line (i.e., flow line of $X_{f}$) $\gamma\colon [0,1]\to M$ from $T_{i}$ to $T_{j}$, we have
		\begin{align*}
			&f\bigl(\gamma(0)\bigr)-f\bigl(\gamma(1)\bigr)\\
			\geq &3\max \left\{\, \var (f,T_{1}),\ldots, \var (f,T_{\ell}), \var (f,T_{1}^{N}),\ldots, \var (f,T_{\ell_N}^{N})\, \right\}.
		\end{align*}
		Moreover, the similar conditions hold in the cases of ``from $T_{i}$ to $T_{j}^{N}$'',
		``from $T_{i}^{N}$ to $T_{j}$'' and ``from $T_{i}^{N}$ to $T_{j}^{N}$''.
	\item If $f(C_{i})\neq f(C_{j})$, then
		$\var (f,T_{i}) + \var (f,T_{j})< 1/3\,\bigl\lvert f(C_{i})-f(C_{j}) \bigr\rvert$.
		Moreover, the similar conditions hold in the cases of $f(\Gamma_{i})\neq f(\Gamma_{j})$ and $f(C_{j})\neq f(\Gamma_{s})$.
	\item For all $j$, we have $\bigl((0,1)\times \partial M \bigr)\cap T_{j} = \emptyset$.
\end{enumerate}
Similarly, for all $u=1,\ldots,\ell_D$, let $T_{u}^{D}\subset \partial M$
be a tubular neighborhood of $\Delta_{u}$ satisfying the above conditions (2) and (3).

By the Kupka--Smale Theorem \cite{Kupka, Smale}, for each $j$, $s$ and $u$, we can pick positive Morse functions
$f_{j}\colon C_{j}\to \mathbb{R}$,\: $f_{s}^{N}\colon \Gamma_{s}\to \mathbb{R}$ and $f_{u}^{D}\colon \Delta_{u}\to \mathbb{R}$
such that their gradient vector fields satisfy the Morse--Smale condition, i.e.,
all the unstable manifolds and the stable manifolds intersect transversely.
We extend these functions to functions on $T_{j}$, $T_{s}^{N}$ and $T_{u}^{D}$
by making constant in the direction normal to $C_{j}$, $\Gamma_{s}$ and $\Delta_{u}$ respectively.
For all $j=1,\ldots,\ell$, let $\tilde{T}_{j}\subset T_{j}$ be a smaller tubular neighborhood of $C_{j}$
of radius $\tilde{\delta}_{j}\ (<\delta_{j})$ in $\Int M$.

Moreover, let $\hat{\rho}_{j}\colon [0,\infty )\to [0,1]$ be a $C^{\infty}$-function satisfying
\begin{align*}
\hat{\rho}_j &=%
\begin{cases}
1 & \text{on\: $[0,\tilde{\delta}_{j})$},\\
0 & \text{outside of\: $[0,\delta_{j})$},
\end{cases}\\
\hat{\rho}_j&>0\quad \text{on\: $[0,\delta_{j})$}
\end{align*}
and we define a bump function $\rho_{j}$ by the formula $\rho_{j}(x)=\hat{\rho}_{j}\bigl(d(x,C_{j})\bigr)$
where $d$ is the distance induced from a metric on $M$.
Now we choose $\varepsilon_{1}>0$ small enough so that for all $j=1,\ldots,\ell$,
\[
	\sup_{T_{j}\setminus \tilde{T}_{j}} \varepsilon_{1}\big\| \grad (\rho_{j}f_{j})\big\|%
	< \inf_{T_{j}\setminus \tilde{T}_{j}} \big\| \grad f\big\| \neq 0.
\]
Similarly, for every $s$ and $u$, we choose $\tilde{T}_{s}^{N}$, $\tilde{T}_{u}^{D}$,
$\rho_{s}^{N}$, $\rho_{u}^{D}$, $\varepsilon_{1}^{N}$ and $\varepsilon_{1}^{D}$.
For each $s$ and $u$, we extend $f_{s}^{N}$ and $f_{u}^{D}$ to functions on $[0,1)\times T_{s}^{N}$
and $[0,1)\times T_{u}^{D}$ by making constant in the $r$-coordinate, respectively.
Similarly, we extend $\rho_{s}^{N}$ and $\rho_{u}^{D}$ to functions on $[0,1)\times \partial M$.

On the other hand, let $\hat{\rho}\colon [0,\infty )\to [0,1]$ be a $C^{\infty}$-function satisfying
\begin{align*}
\hat{\rho} &=%
\begin{cases}
1 & \text{on\: $[0,2\delta )$}, \\
0 & \text{outside of\: $[0,1)$},
\end{cases}\\
\hat{\rho}&>0\quad \text{on\: $[0,1)$}
\end{align*}
and we define a bump function $\rho$ by the formula $\rho(r,y)=\hat{\rho}(r)$ where $(r,y)\in [0,1)\times \partial M$.
By the condition (1) of Definition \ref{definition:MBwb}, 
we can choose $\varepsilon_{2}^{N}>0$ small enough so that for all $s=1,\ldots,\ell_N$,
\[
	\sup_{[0,1)\times T_{s}^{N}} \varepsilon_{2}^{N}\big\| \grad \bigl(\rho \rho_{s}^{N}f_{s}^{N}\bigr)\big\|%
	< \inf_{[0,1)\times T_{s}^{N}} \big\| \grad f\big\| \neq 0.
\]
We can choose $\varepsilon_{2}^{D}$ satisfying a similar condition.

Lastly, we set
$\varepsilon = \min \{\, \varepsilon_{1}, \varepsilon_{1}^{N}, \varepsilon_{1}^{D}, \varepsilon_{2}^{N}, \varepsilon_{2}^{D} \, \}$
and define a $C^{\infty}$-function on $M$ by
\[
	h= f + \varepsilon \left\{ \sum_{j=1}^{\ell}\rho_{j}f_{j}%
		+ \rho \left(\sum_{s=1}^{\ell_N}\rho_{s}^{N}f_{s}^{N} + \sum_{u=1}^{\ell_D}\rho_{u}^{D}f_{u}^{D} \right) \right\}.
\]

The proof of the following lemma is straightforward:

\begin{lemma}\label{lemma:main}
The function $h\colon M\to \mathbb{R}$ is Morse in the sense of Laudenbach {\rm \cite{Laudenbach}}.
For all $n=0,1,\ldots,m$,
\[
	I_{n} (h)= \bigsqcup_{\lambda_{j}+k=n}\Cr_{k} (f_{j}),
	\quad N_{n} (h)= \bigsqcup_{\mu_{s}^{N}+k=n}\Cr_{k} (f_{s}^{N})\\
\]
and
\[
	\quad D_{n} (h)= \bigsqcup_{\mu_{u}^{D}+k=n}\Cr_{k} (f_{u}^{D}).
\]
\end{lemma}

Since the number of critical points is finite,
we can choose $\tilde{\delta}>0$ small enough so that
for all $s=1,\ldots,\ell_N$ and $\gamma\in N(h)\cap \Gamma_s$,
the open ball $B_{4\tilde{\delta}}(\gamma)\subset T_s^N$ of $\gamma$ of radius $4\tilde{\delta}$ contains no other critical points.
Let $\hat{\tilde{\rho}}\colon [0,\infty )\to [0,1]$ be a $C^{\infty}$-function satisfying
\begin{align*}
\hat{\tilde{\rho}} &=%
\begin{cases}
0 & \text{on\: $[0,\tilde{\delta})$}, \\
1 & \text{outside of\: $[0,2\tilde{\delta})$},
\end{cases}\\
\hat{\tilde{\rho}}&>0\quad \text{on\: $[0,2\tilde{\delta})$}
\end{align*}
and we define a bump function $\tilde{\rho}$ by the formula $\tilde{\rho}(y)=\hat{\tilde{\rho}}\Bigl(d\bigl(y,N(h)\bigr)\Bigr)$
where $y\in \partial M$.
Then we extend $\tilde{\rho}$ to a function on $[0,1)\times \partial M$ by making constant in the $r$-coordinate.

Now let us choose $\varepsilon>0$ small enough so that for all $s=1,\ldots,\ell_N$,
\begin{equation}
	\sup_{[0,2\delta)\times \left(T_{s}^{N}\setminus B_{\tilde{\delta}}^s\right)}%
	\varepsilon\delta\left\lvert \left(\grad (\tilde{\rho}r)\right)h\right\rvert%
	< \inf_{T_{s}^{N}\setminus B_{\tilde{\delta}}^s} \big\| \grad h|_{\partial M}\big\|^2 \neq 0\label{eq:4}
\end{equation}
and
\begin{align}
	&\sup_{[2\delta,1)\times T_{s}^{N}} \varepsilon\left\lvert X_{f} \bigl(\rho \rho_{s}^{N}f_{s}^{N}\bigr)%
		- (\grad f)\bigl(\rho \rho_{s}^{N}f_{s}^{N}\bigr) + \delta\left(\grad (\rho\tilde{\rho}r)\right)h \right\rvert\notag\\
	<&\inf_{[2\delta,1)\times T_{s}^{N}} \left\lvert X_{f} f \right\rvert \neq 0,\label{eq:5}
\end{align}
where
\[
	B_{\tilde{\delta}}^s=\bigcup_{\gamma\in N(h)\cap \Gamma_s}B_{\tilde{\delta}}(\gamma).
\]
We define a $C^{\infty}$-vector field on $M$ by
\[
	G= X_{f}-\varepsilon \grad \left\{ \sum_{j=1}^{\ell}\rho_{j}f_{j}%
		+ \rho \left(\sum_{s=1}^{\ell_N}\rho_{s}^{N}f_{s}^{N} + \sum_{u=1}^{\ell_D}\rho_{u}^{D}f_{u}^{D}%
		-\delta\tilde{\rho}r\right) \right\}.
\]

\begin{lemma}
The vector field $G$ is an adapted pseudo-gradient vector field for the Morse function $h$.
\end{lemma}

\begin{proof}
We show that $G$ satisfies the conditions (1)--(5) of the definition of
the adapted pseudo-gradient vector fields of \cite{Laudenbach},
that is those of Definition \ref{definition:pseudoBott} for a Morse function.

The condition (3) is verified as follows:
For $p\in I(h)$, there exists $j=1,\ldots,\ell$ such that $p\in C_{j}\subset T_{j}$.
Since $h=f+\varepsilon \rho_{j}f_{j}$ and $X_{f} = -\grad f$ on $T_{j}$ by \eqref{eq:1}, we have
\begin{equation}
	G= X_{f} - \varepsilon \grad (\rho_{j}f_{j})= -\grad f - \varepsilon \grad (\rho_{j}f_{j})= -\grad h.\label{eq:6}
\end{equation}
Hence the Hessian of the Lie derivative $G\cdot h$ at $p$ is negative definite
since $p$ is a non-degenerate critical point of $h$ by Lemma \ref{lemma:main}.

The condition (4) is verified as follows:
For $\gamma\in N(h)$, there exists $s=1,\ldots,\ell_N$ such that $\gamma\in \Gamma_{s}\subset T_{s}^{N}$.
By Lemma \ref{lemma:main} and \eqref{eq:2}, on $[0,2\delta)\times T_{s}^{N}$,
there exist local coordinates
$x=(r,y)\in [0,1)\times \mathbb{R}^{m-1}$ around $\gamma$ in $M$
and a non-degenerate quadratic form $q$ such that 
\begin{equation}
	h(x)=h(\gamma)+q(y)+r.\label{eq:7}
\end{equation}
By \eqref{eq:3}, on $[0,2\delta)\times T_{s}^{N}$, we obtain
\begin{align}
	G &=X_{f} - \varepsilon \grad \bigl(\rho_{s}^{N}f_{s}^{N}-\delta\tilde{\rho}r\bigr)\notag \\
	&= \left(-\grad f|_{\partial M}-r \frac{\partial}{\partial r}\right)
	- \varepsilon \grad \bigl(\rho_{s}^{N}f_{s}^{N}\bigr)+ \varepsilon\delta \grad \bigl(\tilde{\rho}r\bigr)\notag \\
	&= -\grad h|_{\partial M}- r \frac{\partial}{\partial r}+ \varepsilon\delta \grad \bigl(\tilde{\rho}r\bigr)\label{eq:8}
\end{align}
and in particular, on $[0,2\delta)\times B_{\tilde{\delta}}(\gamma)$, we have
\begin{equation}
	G = -\grad h|_{\partial M}- r \frac{\partial}{\partial r}
	= -\grad q- r \frac{\partial}{\partial r}.\label{eq:9}
\end{equation}
Hence the Hessian of the Lie derivative $G\cdot h$ is negative definite.

The condition (2) is verified as follows:
By \eqref{eq:9}, since $G = -\grad h|_{\partial M}$ on $\{0 \} \times B_{\tilde{\delta}}(\gamma)$,
$G$ is tangent to the boundary on the neighborhood $B_{\tilde{\delta}}(\gamma)$
of a boundary critical point $\gamma$ of type $N$.
On $\partial M\setminus \bigcup_s B_{\tilde{\delta}}^s$, we have
\[
	G=X_{f}-\varepsilon \grad%
		\left(\sum_{s=1}^{\ell_N}\rho_{s}^{N}f_{s}^{N} + \sum_{u=1}^{\ell_D}\rho_{u}^{D}f_{u}^{D}\right)%
		+\delta\left(r\grad\tilde{\rho}+\tilde{\rho}\frac{\partial}{\partial r}\right).
\]
Due to the term $\delta\tilde{\rho}\,\partial/\partial r\neq 0$, $G$ points inwards along the boundary outside of a neighborhood containing $\bigcup_s B_{\tilde{\delta}}^s$.

The verification of the condition (1) divides into five cases:
On a neighborhood $U_{p}\setminus \{ p\}$ of a critical point $p\in I(h)\cap C_j\subset T_j$,
we have $G=-\grad h$ by \eqref{eq:6} and $Gh=-(\grad h)h=-\big\| \grad h \big\|^{2}<0$.
On the neighborhood $[0,2\delta)\times B_{\tilde{\delta}}(\gamma)$ of a critical point $\gamma\in N(h)\cap \Gamma_s\subset T_s^N$,
we have $h=h|_{\partial M}+r$ by \eqref{eq:7}
and
\[
	Gh= \left(-\grad h|_{\partial M}- r \frac{\partial}{\partial r}\right)(h|_{\partial M}+r)
	=-\big\| \grad h|_{\partial M} \big\|^{2}-r<0
\]
by \eqref{eq:9}.
On $[0,2\delta)\times \left(T_s^N\setminus B_{\tilde{\delta}}^s\right)$,
we have
\begin{align*}
	Gh&= \left(-\grad h|_{\partial M}- r \frac{\partial}{\partial r}+ \varepsilon\delta \grad \bigl(\tilde{\rho}r\bigr)\right)(h|_{\partial M}+r)\\
	&=-\big\| \grad h|_{\partial M} \big\|^{2}-r+\varepsilon\delta \left(\grad \bigl(\tilde{\rho}r\bigr)\right)h\\
	&< -\inf_{T_{s}^{N}\setminus B_{\tilde{\delta}}^s} \big\| \grad h|_{\partial M}\big\|^2%
		+\sup_{[0,2\delta)\times \left(T_{s}^{N}\setminus B_{\tilde{\delta}}^s\right)}%
		\varepsilon\delta\left\lvert \left(\grad (\tilde{\rho}r)\right)h\right\rvert<0
\end{align*}
by \eqref{eq:7}, \eqref{eq:8} and \eqref{eq:4}.
On a neighborhood $[2\delta,1)\times U_{\gamma}$ of the critical point $\gamma$,
we have $h= f + \varepsilon \rho\rho_{s}^{N}f_{s}^{N}$ and
\begin{align*}
	Gh&= \left(X_{f} - \varepsilon \grad \bigl(\rho\rho_{s}^{N}f_{s}^{N}\bigr)+\varepsilon\delta \grad \bigl(\rho\tilde{\rho}r\bigr)\right)%
		\left(f + \varepsilon \rho\rho_{s}^{N}f_{s}^{N}\right)\\
	&= X_{f}f +\varepsilon \left\{ X_{f}\bigl(\rho \rho_{s}^{N} f_{s}^{N}\bigr) - \left(\grad \bigl(\rho \rho_{s}^{N} f_{s}^{N}\bigr)\right)f \right\}\\
	&\qquad - \varepsilon^{2} \big\| \grad \bigl(\rho \rho_{s}^{N} f_{s}^{N}\bigr) \big\|^{2}%
		+\varepsilon\delta \left(\grad \bigl(\rho\tilde{\rho}r\bigr)\right)h\\
	&< -\inf_{[2\delta,1)\times T_{s}^{N}} \bigl\lvert X_{f} f \bigr\rvert%
		+\varepsilon \left\{ X_{f}\bigl(\rho \rho_{s}^{N} f_{s}^{N}\bigr)
		-(\grad f)\bigl(\rho \rho_{s}^{N} f_{s}^{N}\bigr)+\delta \left(\grad \bigl(\rho\tilde{\rho}r\bigr)\right)h\right\}\\
	&< -\inf_{[2\delta,1)\times T_{s}^{N}} \bigl\lvert X_{f} f \bigr\rvert\\
	&\qquad+ \sup_{[2\delta,1)\times T_{s}^{N}} \varepsilon\left\lvert X_{f} \bigl(\rho \rho_{s}^{N}f_{s}^{N}\bigr)%
		- (\grad f)\bigl(\rho \rho_{s}^{N} f_{s}^{N}\bigr)+\delta \left(\grad \bigl(\rho\tilde{\rho}r\bigr)\right)h \right\rvert< 0.
\end{align*}
Here we used the fact that $X_{f}f<0$ and \eqref{eq:5}.
Outside of all the neighborhoods we have considered till here, we have $h=f$ and $G=X_{f}$.
This implies that $G h=X_{f} f< 0$.

According to \cite{S}, the condition (5) is generically fulfilled for vector fields satisfying conditions (1)--(4).
Thus we choose $\varepsilon$ small enough so that $G$ satisfies the condition (5) if necessary.
\end{proof}


\subsection{Lemmas}

Let $\partial_{\ast}^h$ be the boundary operator of the Morse complex for the Morse function $h$ defined in \cite{Laudenbach}.
For $j=1,\ldots,\ell$ and $s=1,\ldots,\ell_N$, let $\partial_{\ast}^j=\partial_{\ast}^{f_j;o(\nu^-C_j)}$ and $\partial_{\ast}^{s,N}=\partial_{\ast}^{f_s^N;o(\nu^-\Gamma_s)}$
be the boundary operators for the Morse functions $f_{j}$ and $f_{s}^{N}$ with local coefficients in $o(\nu^-C_j)$ and $o(\nu^-\Gamma_s)$
defined in Section \ref{section:preliminaries}, respectively.

Let us denote by $o(\nu^-C_j)_x\cong\mathbb{Z}$ the fiber of $o(\nu^-C_j)$ over a point $x\in C_j$.
We fix a base point $\ast\in C_j$.
For every $p\in\Cr(f_j)$, let $c_p\colon [0,1]\to C_j$ be a path such that $c_p(0)=\ast$ and $c_p(1)=p$.
Since the interval $[0,1]$ is contractible and then the pull-back bundle $c_p^{\ast}o(\nu^-C_j)\to [0,1]$ is trivial,
we can choose bases $1_p$ of $o(\nu^-C_j)_p$, $p\in\Cr(f_j)$, so that 
the associated isomorphism $\Phi_{c_p}\colon o(\nu^-C_j)_{\ast}\to o(\nu^-C_j)_p$ satisfies $\Phi_{c_p}(1_{\ast})=1_p$.
Similarly, we choose bases $1_p$ of $o(\nu^-\Gamma_s)_p$, $p\in\Cr(f_s^N)$.

We note that an orientation of the unstable manifold $W_{f_j}^u(p)$ of a critical point $p$
is determined by an orientation of the tangent space $T_pW_{f_j}^u(p)$ at $p$.
We fix orientations of the unstable manifolds $W_G^u(p)$, $p\in\Cr(h)$,
where $G$ is the adapted pseudo-gradient vector field for $h$ defined in Subsection \ref{subsection:pt}.
Since we have the direct sum $T_pW_G^u(p)=T_pW_{f_j}^u(p)\oplus\nu_p^-C_j$ at $p$,
we get an orientation of $W_{f_j}^u(p)$.
In the same manner, we choose orientations of $W_{f_s^N}^u(p)$, $p\in\Cr(f_s^N)$.
The following lemmas are generalizations of \cite[Lemma 9 and Corollary 10]{BH}.

\begin{lemma}\label{lemma:1st}
Let $j=1,\ldots,\ell$ and $s=1,\ldots,\ell_N$.
For all critical points $p$, $q\in \Cr (f_{j})$ $($resp.\ $\Cr (f_{s}^{N})$$)$ of relative index one
and all unparametrized $($pseudo-$)$gradient flow lines $\gamma$ from $p$ to $q$,
we have
\[
	n_h(\gamma)1_q=n_{f_j}(\gamma)\Phi_{\tilde{\gamma}}(1_p)
\]
$($resp.\ $n_h(\gamma)1_q=n_{f_s^N}(\gamma)\Phi_{\tilde{\gamma}}(1_p)$$)$.
\end{lemma}

\begin{proof}
We will focus on $C_j$.
By the definition of $h$ (and the choice of $\varepsilon$), the gradient flow lines connecting two critical points in $C_j$ are the same (see the proof of \cite[Lemma 9]{BH}).

The signs $n_h(\gamma)$ and $n_{f_j}(\gamma)$ are determined by the orientations of the unstable manifolds $W_G^u(p)$, $W_G^u(q)$, $W_{f_j}^u(p)$ and $W_{f_j}^u(q)$ chosen above.
We define a loop $l\colon S^1=\mathbb{R}/\mathbb{Z}\to C_j$ by
\[
l(t)=%
\begin{cases}
\tilde{\gamma}(3t) & \text{if\: $0\leq t\leq 1/3$},\\
c_q(2-3t) & \text{if\: $1/3\leq t\leq 2/3$},\\
c_p(3t-2) & \text{if\: $2/3\leq t\leq 1$}.
\end{cases}
\]
Since we have the direct sum $T_xW_G^u(p)=T_xW_{f_j}^u(p)\oplus\nu_x^-C_j$ for every $x\in W_G^u(p)\cap C_j$,
the orientations of $W_G^u(p)$ and $W_{f_j}^u(p)$ is not compatible along the loop $l$
as long as the pull-back bundle $l^{\ast}o(\nu^-C_j)\to S^1$ is non-trivial.
Hence the ambiguity of these orientations appears as the multiplication by $\Phi_l(1_p)$.
Namely, we have
\[
	n_h(\gamma)1_p=n_{f_j}(\gamma)\Phi_l(1_p).
\]
Therefore,
\begin{align*}
	n_h(\gamma)1_q
	&=n_h(\gamma)\Phi_{c_q}\left(\Phi_{c_p}^{-1}\bigl(1_p\bigr)\right)
	=\Phi_{c_q}\left(\Phi_{c_p}^{-1}\bigl(n_h(\gamma)1_p\bigr)\right)\\
	&=\Phi_{c_q}\left(\Phi_{c_p}^{-1}\bigl(n_{f_j}(\gamma)\Phi_l(1_p)\bigr)\right)
	=n_{f_j}(\gamma)\Phi_{\tilde{\gamma}}(1_p).\qedhere
\end{align*}
\end{proof}

Now we order $C_{1}$, $\ldots\,$,~$C_{\ell}$, $\Gamma_{1}$, $\ldots\,$,~$\Gamma_{\ell_N}$ by ``height''.
That is, for
\[
	\{ B_1,\ldots,B_{\ell+\ell_N}\} =\{C_1,\ldots,C_{\ell},\Gamma_1,\ldots,\Gamma_{\ell_N}\},
\]
we assume that the order of $B_{1}$, $\ldots\,$,~$B_{\ell+\ell_N}$ is ascending, i.e.,
$f(B_{i}) \leq f(B_{j})$ whenever $i\leq j$.
On the other hand, by Lemma \ref{lemma:main}, for each $n=0,1,\ldots,m$,
the $n$-chain group $F_{n}^{N}(h)$ of the Morse complex (see \cite{Laudenbach} for the definition) of $h$ is of the form:
\[
	F_{n}^{N}(h)= \bigoplus_{j=1}^{\ell}C_{n-\lambda_{j}}(f_{j};\mathbb{Z})\oplus \bigoplus_{s=1}^{\ell_N}C_{n-\mu_{s}^{N}}(f_{s}^{N};\mathbb{Z}),
\]
where $C_{n-\lambda_{j}}(f_{j};\mathbb{Z})$ and $C_{n-\mu_{s}^{N}}(f_{s}^{N};\mathbb{Z})$ are the free $\mathbb{Z}$-modules generated by
the elements of $\Cr_{n-\lambda_{j}}(f_{j})$ and $\Cr_{n-\mu_{s}^{N}}(f_{s}^{N})$, respectively.
Therefore any $\beta\in F_{n}^{N}(h)\setminus \{0\}$ can be uniquely decomposed into
the sum of $n$-chains of $\Cr_{n-\lambda_{j}}(f_{j})$ and $\Cr_{n-\mu_{s}^{N}}(f_{s}^{N})$, $\beta= \beta_{j_{1}}+\cdots+\beta_{j_{r}}$
where $j_{1}<\cdots< j_{r}$ and for $i=1,\ldots,r$ there exists $j=1,\ldots,\ell$ or $s=1,\ldots,\ell_N$ such that
$\beta_{j_{i}} \in C_{n-\lambda_{j}}(f_{j};\mathbb{Z})\setminus \{0\}$ or $\beta_{j_{i}} \in C_{n-\mu_{s}^{N}}(f_{s}^{N};\mathbb{Z})\setminus \{0\}$ respectively.
We call $\beta_{j_{r}}$ the \textit{top chain} of $\beta$ and denote by $\tp{\beta}$.

\begin{lemma}\label{lemma:2nd}
Let $n=0,1,\ldots,m$.
If $\beta\in \Ker \partial_{n}^{h}$,
then we have
\[
	\mbox{$1$}\hspace{-0.25em}\mbox{$\mathrm{l}$}\otimes\tp{\beta}\in \Ker \partial_{n-\lambda_j}^j%
	\quad\text{or}\quad
	\mbox{$1$}\hspace{-0.25em}\mbox{$\mathrm{l}$}\otimes\tp{\beta}\in \Ker \partial_{n-\mu_{s}^{N}}^{s,N},
\]
where $\mbox{$1$}\hspace{-0.25em}\mbox{$\mathrm{l}$}=\sum_{q\in\tp{\beta}}1_q\in o(\nu^-C_j)$ $($the sum is taken over generators$)$
and $\tp{\beta}\in C_{n-\lambda_{j}}(f_{j};\mathbb{Z})$ for some $j=1,\ldots,\ell$,
or $\mbox{$1$}\hspace{-0.25em}\mbox{$\mathrm{l}$}=\sum 1_q\in o(\nu^-\Gamma_s)$
and $\tp{\beta}\in C_{n-\mu_{s}^{N}}(f_{s}^{N};\mathbb{Z})$ for some $s=1,\ldots,\ell_N$, respectively.
\end{lemma}

\begin{proof}
Let $\beta=\sum_in_iq_i$ where $n_i\in\mathbb{Z}$ and $q_i\in\Cr_n(h)$.
If $\beta\in \Ker \partial_{n}^{h}$, we then have
\begin{align*}
	0&=\partial_n^h(\beta)=\sum_in_i\partial_n^h(q_i)%
	=\sum_in_i\left\{\sum_{p\in I_{n-1}(h)\cup N_{n-1}(h)}\left(\sum_{\gamma\in\mathcal{M}_G(q_i,p)}n_h(\gamma)\right)p\right\}\\
	&=\sum_{p\in I_{n-1}(h)\cup N_{n-1}(h)}\left\{\sum_in_i\left(\sum_{\gamma\in\mathcal{M}_G(q_i,p)}n_h(\gamma)\right)\right\}p.
\end{align*}
Hence for every $p\in I_{n-1}(h)\cup N_{n-1}(h)$ we have
\[
	\sum_in_i\left(\sum_{\gamma\in\mathcal{M}_G(q_i,p)}n_h(\gamma)\right)=0.
\]
From now on, we will focus on $C_j$'s.
Namely, we assume that $\tp{\beta}\in C_{n-\lambda_{j}}(f_{j};\mathbb{Z})$ for some $j=1,\ldots,\ell$.
Then, in particular, for every $p\in I_{n-1}(h)\cap C_j=\Cr_{n-\lambda_j-1}(f_j)$ we have
\[
	0=\sum_in_i\left(\sum_{\gamma\in\mathcal{M}_G(q_i,p)}n_h(\gamma)\right)=\sum_{q_i\in C_j}n_i\left(\sum_{\gamma\in\mathcal{M}_G(q_i,p)}n_h(\gamma)\right)
\]
since $\mathcal{M}_G(q_i,p)=\emptyset$ for every $q_i\not\in C_j$ by the definition of top chains.
Therefore, we have
\begin{align*}
	\partial_{n-\lambda_j}^j(\1\otimes\tp{\beta})&=\sum_{q_i\in C_j}n_i\partial_{n-\lambda_j}^j(1_{q_i}\otimes q_i)\\
	&=\sum_{q_i\in C_j}n_i\left\{\sum_{p\in \Cr_{n-\lambda_j-1}(f_j)}\left(\sum_{\gamma\in\mathcal{M}_{f_j}(q_i,p)}n_{f_j}(\gamma)\Phi_{\tilde{\gamma}}(1_{q_i})\right)\otimes p\right\}\\
	&=\sum_{p\in \Cr_{n-\lambda_j-1}(f_j)}\left\{\sum_{q_i\in C_j}n_i\left(\sum_{\gamma\in\mathcal{M}_{f_j}(q_i,p)}n_{f_j}(\gamma)\Phi_{\tilde{\gamma}}(1_{q_i})\right)\right\}\otimes p\\
	&=\sum_{p\in I_{n-1}(h)\cap C_j}\left\{\sum_{q_i\in C_j}n_i\left(\sum_{\gamma\in\mathcal{M}_G(q_i,p)}n_h(\gamma)1_p\right)\right\}\otimes p\\
	&=\sum_{p\in I_{n-1}(h)\cap C_j}\left\{\sum_{q_i\in C_j}n_i\left(\sum_{\gamma\in\mathcal{M}_G(q_i,p)}n_h(\gamma)\right)\right\}(1_p\otimes p)\\
	&=0.
\end{align*}
Here we applied Lemma \ref{lemma:1st} and used the fact that $\mathcal{M}_G(q_i,p)$ coincides with $\mathcal{M}_{f_j}(q_i,p)$ whenever $q_i$, $p\in C_j$.
\end{proof}


\subsection{Proof of Main Theorem}

Under these preparations, the remaining part of the proof
is based on a modification of \cite{BH}.

\begin{proof}[Proof of Theorem \ref{theorem:main}]
Applying Theorem \ref{theorem:mi} for the Morse functions $f_{1}$, $\ldots\,$,~$f_{\ell}$ and $f_{1}^{N}$, $\ldots\,$,~$f_{\ell_N}^{N}$,
there exist polynomials $\mathcal{R}_{1}(t)$, $\ldots\,$,~$\mathcal{R}_{\ell}(t)$ and
$\mathcal{R}_{1}^{N}(t)$, $\ldots\,$,~$\mathcal{R}_{\ell_N}^{N}(t)$ with non-negative integer coefficients,
such that for $j=1,\ldots,\ell$ and $s=1,\ldots,\ell_N$,
\[
	\mathcal{M}_t(f_j)-\mathcal{P}_t\bigl(C_j;o(\nu^-C_j)\bigr)=(1+t)\mathcal{R}_j(t)
\]
and
\[
	\mathcal{M}_t(f_s^N) - \mathcal{P}_t\bigl(\Gamma_s;o(\nu^-\Gamma_s)\bigr)=(1+t)\mathcal{R}_s^N(t),
\]
respectively.
Here we note that $\rank{o(\nu^-C_j)}=\rank{o(\nu^-\Gamma_s)}=1$.
Moreover, by \cite[Corollary A]{Laudenbach}, for the Morse function $h$,
there exists a polynomial $\mathcal{R}_{h}(t)$ with non-negative integer coefficients such that
\[
	\mathcal{M}_{t}^{N}(h) - \mathcal{P}_{t}(M;\mathbb{Z})= (1+t)\mathcal{R}_{h}(t),
\]
where
\[
	\mathcal{M}_{t}^{N}(h)= \sum_{p\in I(h)} t^{\ind_{h}p} + \sum_{\gamma\in N(h)} t^{\ind_{h|_{\partial M}} \gamma}
\]
is the Morse counting polynomial of type $N$ of $h$.
By Lemma \ref{lemma:main}, $\mathcal{M}_{t}^{N}(h)$ can be computed as follows:
\begin{align*}
	\mathcal{M}_{t}^{N}(h)%
	&= \sum_{j,\, k} \# \Cr_{k} (f_{j})\, t^{\lambda_{j}+k} + \sum_{s,\, k} \# \Cr_{k} (f_{s}^{N})\, t^{\mu_{s}^{N}+k}\\
	&= \sum_{j=1}^{\ell} \mathcal{M}_{t}(f_{j})\, t^{\lambda_{j}} + \sum_{s=1}^{\ell_N} \mathcal{M}_{t}(f_{s}^{N})\, t^{\mu_{s}^{N}}.
\end{align*}
Hence we have
\begin{align*}
	\mathcal{MB}^{N}_{t}(f)%
	&=\sum_{j=1}^{\ell} \mathcal{P}_{t}\bigl(C_{j};o(\nu^-C_j)\bigr)\, t^{\lambda_{j}} + \sum_{s=1}^{\ell_N} \mathcal{P}_{t}\bigl(\Gamma_{s};o(\nu^-\Gamma_{s})\bigr)\, t^{\mu_{s}^{N}}\\
	&= \sum_{j=1}^{\ell} \left\{ \mathcal{M}_{t}(f_{j}) - (1+t)\mathcal{R}_{j}(t) \right\} t^{\lambda_{j}}%
	+ \sum_{s=1}^{\ell_N} \left\{ \mathcal{M}_{t}(f_{s}^{N}) - (1+t)\mathcal{R}_{s}^{N}(t) \right\} t^{\mu_{s}^{N}}\\
	&= \mathcal{M}_{t}^{N}(h)%
		- (1+t) \left( \sum_{j=1}^{\ell} \mathcal{R}_{j}(t)\, t^{\lambda_{j}} + \sum_{s=1}^{\ell_N} \mathcal{R}_{s}^{N}(t)\, t^{\mu_{s}^{N}} \right)\\
	&= \mathcal{P}_{t}(M;\mathbb{Z}) + (1+t) \left\{ \mathcal{R}_{h}(t) - \left( \sum_{j=1}^{\ell} \mathcal{R}_{j}(t)\, t^{\lambda_{j}}%
	+ \sum_{s=1}^{\ell_N} \mathcal{R}_{s}^{N}(t)\, t^{\mu_{s}^{N}} \right) \right\}.
\end{align*}
We set
$\mathcal{R}(t) = \mathcal{R}_{h}(t) - \left( \sum_{j=1}^{\ell} \mathcal{R}_{j}(t)\, t^{\lambda_{j}} + \sum_{s=1}^{\ell_N} \mathcal{R}_{s}^{N}(t)\, t^{\mu_{s}^{N}} \right)$.
It is enough to show that all the coefficients of the polynomial $\mathcal{R}(t)$ are non-negative.

According to the proof of \cite[Theorem 3]{BH},
$\mathcal{R}_{1}(t)$, $\ldots\,$,~$\mathcal{R}_{\ell}(t)$, $\mathcal{R}_{1}^{N}(t)$,
$\ldots\,$,~$\mathcal{R}_{\ell_N}^{N}(t)$ and $\mathcal{R}_{h}(t)$ have the following specific forms:
For all $j=1,\ldots,\ell$ and $s=1,\ldots,\ell_N$,
\begin{gather*}
	\mathcal{R}_{j}(t)= \sum_{k=1}^{c_{j}} \left(\nu_{k}^{j}-z_{k}^{j} \right)t^{k-1},%
	\quad \mathcal{R}_{s}^{N}(t)= \sum_{k=1}^{d_{s}^{N}} \left(\nu_{k}^{s,N}-z_{k}^{s,N} \right)t^{k-1}\\
	\text{and}\qquad\mathcal{R}_{h}(t)= \sum_{n=1}^{m} \left(\nu_{n}+\nu_{n}^{N}-z_{n} \right)t^{n-1},
\end{gather*}
where
\[
	\nu_{k}^{j}=\# \Cr_{k} (f_{j}),%
	\quad z_{k}^{j}= \rank_{\mathbb{Z}}\Ker \partial_{k}^{j},
\]
\[
	\quad \nu_{k}^{s,N}=\# \Cr_{k} (f_{s}^{N}),%
	\quad z_{k}^{s,N}= \rank_{\mathbb{Z}}\Ker \partial_{k}^{s,N},
\]
\[
	\nu_{n}=\# I_{n}(h),%
	\quad \nu_{n}^{N}=\# N_{n}(h),%
	\quad z_{n}= \rank_{\mathbb{Z}}\Ker \partial_{n}^{h}.
\]
Therefore
\begin{align*}
	\mathcal{R}(t)%
	&= \sum_{n=1}^{m} \left(\nu_{n}+\nu_{n}^{N}-z_{n} \right)t^{n-1}\\
	&\qquad- \left\{ \sum_{j=1}^{\ell} \sum_{k=1}^{c_{j}} \left(\nu_{k}^{j}-z_{k}^{j} \right)t^{\lambda_{j}+k-1}%
		+ \sum_{s=1}^{\ell_N} \sum_{k=1}^{d_{s}^{N}} \left(\nu_{k}^{s,N}-z_{k}^{s,N} \right)t^{\mu_{s}^{N}+k-1} \right\}\\
	&= \left( \sum_{j=1}^{\ell} \sum_{k=1}^{c_{j}} z_{k}^{j} t^{\lambda_{j}+k-1}%
		+ \sum_{s=1}^{\ell_N} \sum_{k=1}^{d_{s}^{N}} z_{k}^{s,N} t^{\mu_{s}^{N}+k-1} \right) - \sum_{n=1}^{m} z_{n} t^{n-1}\\
	&= \sum_{n=1}^{m} \left( \sum_{\lambda_{j}+k=n} z_{k}^{j} + \sum_{\mu_{s}^{N}+k=n} z_{k}^{s,N} - z_{n} \right) t^{n-1}.
\end{align*}
Thus it is enough to show that for all $n=1,\ldots,m$,
\[
	\sum_{\lambda_{j}+k=n} z_{k}^{j} + \sum_{\mu_{s}^{N}+k=n} z_{k}^{s,N} \geq z_{n}.
\]

Fix $n=1,\ldots,m$. If $z_{n}=0$, the inequality holds.
Hence we may assume that $z_{n}>0$.

If $z_{n}=1$, there exists a non-zero element $\beta_{1}$ in $\Ker \partial_{n}^{h}$.
By Lemma \ref{lemma:2nd}, $\1\otimes\tp \beta_{1}\in \Ker \partial_{k_{1}}^{j_{1}}$ (or $\1\otimes\tp \beta_{1}\in \Ker \partial_{k_{1}}^{j_{1},N}$)
where $k_{1}$ and $j_{1}$ are integers satisfying $\lambda_{j_{1}}+k_{1}=n$ (or $\mu_{j_{1}}^{N}+k_{1}=n$).
Hence the inequality holds.

If $z_{n}=2$, we can find an element $\beta_{2}\in \Ker \partial_{n}^{h}$
which is not in the group generated by $\beta_{1}$,
and by adding a multiple of $\beta_{1}$ to $\beta_{2}$ if necessary,
we can choose $\beta_{2}$ whose top part $\1\otimes\tp \beta_{2}\in \Ker \partial_{k_{2}}^{j_{2}}$
(or $\1\otimes\tp \beta_{2}\in \Ker \partial_{k_{2}}^{j_{2},N}$) is not in the group generated by $\1\otimes\tp \beta_{1}$
where $k_{2}$ and $j_{2}$ are integers satisfying $\lambda_{j_{2}}+k_{2}=n$ (or $\mu_{j_{2}}^{N}+k_{2}=n$).
Hence the inequality holds.

Repeating this argument finitely many times,
we get generators in $\Ker \partial_{n}^{h}$ whose the tensor products of $\1$ and top parts are linearly independent in
\[
	\bigoplus_{\lambda_{j}+k=n} \Ker \partial_{k}^{j}\oplus \bigoplus_{\mu_{s}^{N}+k=n} \Ker \partial_{k}^{s,N}.
\]
Thus
\[
	\sum_{\lambda_{j}+k=n} \rank_{\mathbb{Z}}\Ker \partial_{k}^{j} + \sum_{\mu_{s}^{N}+k=n} \rank_{\mathbb{Z}}\Ker \partial_{k}^{s,N}\geq \rank_{\mathbb{Z}}\Ker \partial_{n}^{h}.\qedhere
\]
\end{proof}


\subsection*{Acknowledgement}

The author would like to express his sincere gratitude to his supervisor Professor Takashi Tsuboi for his practical advice and to the referee for very important remarks concerning local coefficients.

This work was supported by JSPS KAKENHI Grant Number JP02607057 and the Program for Leading Graduate Schools, MEXT, Japan.
The author was supported by the Grant-in-Aid for JSPS fellows.


\end{document}